\documentclass[11 pt, twoside]{amsart}
\usepackage{parskip}
\usepackage{amsthm,amsmath,amssymb,oldgerm}
\usepackage[a4paper, margin=2.5cm]{geometry}
\usepackage{mathrsfs}
\usepackage{verbatim}
\usepackage{hyperref}
\usepackage{xcolor}
\theoremstyle{plain}
\newtheorem{thm}{Theorem}[section]
\newtheorem*{theorem*}{Main theorem}

\theoremstyle{definition}
\theoremstyle{definition}
\newtheorem{rem}[thm]{Remark}
\let\Im\relax

\let\H\relax
\let\Z\relax
\let\Q\relax

\let\dh\relax

\DeclareMathOperator{\Im}{\mathrm{Im}}  

\DeclareMathOperator{\Z}{\mathbb{Z}}

\DeclareMathOperator{\N}{\mathbb{N}}
\DeclareMathOperator{\caln}{\mathcal{N}}

\DeclareMathOperator{\Q}{\mathbb{Q}}

\DeclareMathOperator{\calm}{\mathcal{M}}
\DeclareMathOperator{\calr}{\mathcal{R}}
\DeclareMathOperator{\R}{\mathbb{R}} 
\DeclareMathOperator{\cals}{\mathcal{S}}

\DeclareMathOperator{\cala}{\mathcal{A}}

\DeclareMathOperator{\calc}{\mathcal{C}}

\DeclareMathOperator{\H}{\mathbb{H}}
\DeclareMathOperator{\C}{\mathbb{C}}
\DeclareMathOperator{\calf}{\mathcal{F}}
\DeclareMathOperator{\G}{\Gamma} 
\DeclareMathOperator{\hypx}{\mu_{\mathrm{hyp}}} 
 
\DeclareMathOperator{\cald}{\mathcal{D}_{\mathrm{hyp}}}

\DeclareMathOperator{\dk}{{\it{d}}_{{\it{k}}}}

\DeclareMathOperator{\tnc}{\mathcal{T}_{{\it{n}}}^{\mathrm{cusp}}}
\DeclareMathOperator{\tmc}{\mathcal{T}_{{\it{m}}}^{\mathrm{cusp}}}
\DeclareMathOperator{\tmndc}{\mathcal{T}_{{\frac{\it{mn}}{d^2} }}^{\mathrm{cusp}}}

\DeclareMathOperator{\tmndcw}{\mathcal{T}_{{\frac{{\it{mn}}}{{\it{d}}^{2}},{\it{2}}}}^{\mathrm{cusp}}}
\DeclareMathOperator{\tmndlw}{\mathcal{T}_{{\frac{{\it{mn}}}{{\it{d}}^{2}},{\it{2}}}}^{{\it{L}}^2}}
\DeclareMathOperator{\vx}{\mathrm{vol}_{\mathrm{hyp}}}
\DeclareMathOperator{\tnl}{\mathcal{T}_{\it{n}}^{{\it{L}}^2}}

\DeclareMathOperator{\Sk}{\mathcal{S}^{{\it{k}}}(\G)}

\DeclareMathOperator{\bk}{\mathcal{B}_{\G}^{{\it{k}}}}
\DeclareMathOperator{\dh}{d_{\mathrm{hyp}}}
\title[Estimates of cusp forms]{Estimates of cusp forms for certain cocompact arithmetic subgroups }
\author{Anilatmaja Aryasomayajula}
\address{Department of Mathematics, Indian Institute of Science Education and Research (IISER) Tirupati, 
Transit campus at Sri Rama Engineering College, Karkambadi Road,
Mangalam (B.O),Tirupati-517507, India.}
\email{anil.arya@iisertirupati.ac.in}
\author{Baskar Balasubramanyam}
\address{Department of Mathematics, Indian Institute of Science Education and Research (IISER) Pune,
Dr. Homi Bhabha Road, Pashan, Pune 411008, India.
}
\email{baskar@iiserpune.ac.in}
\date{\today}
\begin{document}
\begin{abstract}
In this article, we compute estimates of Hecke eigen cusp forms, associated to a certain cocompact Fuchsian subgroup. Let $\G\subset \mathrm{PSL}_{2}(\R)$ be a co-compact Fuchsian subgroup associated to a division quaternion algebra $\cala$ defined over $\Q$. Let $X:=\G\backslash \H$ denote the quotient space, which admits the structure of a compact hyperbolic Riemann surface. Let $\Sk$ denote the complex vector space of cusp forms of weight-$k$ with respect to $\G$, and let $|\cdot|_{\mathrm{pet}}$ denote then point-wise Petersson norm on $\Sk$. Then, for $k\gg1$, and any $f\in\Sk$, a Hecke eigen cusp form, which is normalized with respect to the Petersson inner-product on $\Sk$, and for any $\epsilon >0$, we derive the following estimate
\begin{equation*}
\sup_{z\in X}\big|f(z)\big|_{\mathrm{pet}}=O_{\cala,\epsilon}\big(k^{\frac{1}{2}-\frac{1}{12}+\epsilon}\big),
\end{equation*}
where the implied constant depends on the quaternion algebra $\cala$, and on the choice of $\epsilon>0$.
\end{abstract}
\vspace{0.2cm}

\maketitle
\section{Introduction}\label{intro}
\subsection{History and Background}\label{subsec1.1}
Estimates of automorphic forms has been a subject of great interest in the recent past. In this article, we derive estimates of holomorphic Hecke eigen cusp forms, and improve upon existing estimates of Das and Sengupta from \cite{das-sen}.  

We now briefly elucidate the history of the problem. Let $\G$ be a co-compact arithmetic subgroup arising from the unit group of a quaternion division algebra (see section \ref{subsec-2.1} for precise definition of $\G$). Let $X:=\G\backslash \H$ be the quotient space, which admits the structure of a hyperbolic Riemann surface of finite hyperbolic volume, and let $\cald$ denote the hyperbolic Laplacian on $X$, which acts on smooth functions defined on $X$.  

Let $\varphi$ be a Hecke eigen Maass form with Laplacian eigenvalue $\lambda$, and we assume that $\varphi$ is $L^{2}$-normalized. With hypothesis as above, for any 
$\epsilon>0$, in a seminal article \cite{iwaniec-sarnak}, Iwaniec and Sarnak proved the following estimate
\begin{align}\label{sarnak-estimate}
\sup_{z\in X}\big|\varphi(z \big|=O_{\G,\epsilon}\big(\lambda^{\frac{5}{24}+\epsilon}\big),
\end{align}
where the implied constant depends on the Fuchsian subgroup $\G$, and on the choice of $\epsilon>0$. 

Let hypothesis be as above, i.e., $\G$ is a co-compact arithmetic subgroup as described above, or a congruence subgroup of level $N$, and let $\varphi$ be a Hecke 
eigen Maass form, which is $L^{2}$-normalized. Then, in 2004, in a famous letter  \cite{letter-sarnak}, for and any $\epsilon>0$, Sarnak conjectured the following estimate
\begin{align}\label{sarnak-conj}
\sup_{z\in X}\big|\varphi(z) \big|=O_{\G,\epsilon}\big( \lambda^{\epsilon}\big),
\end{align}
where the implied constant depends on the Fuchsian subgroup $\G$, and on the choice of $\epsilon>0$.

\vspace{0.1cm}
In 2018, in \cite{pablo}, Ramacher and Wakatsuki have extended estimate \eqref{sarnak-estimate} to compact arithmetic quotients of semi simple Lie groups. 

\vspace{0.1cm}
Sarnak's conjecture \eqref{sarnak-conj} has inspired and instigated deep investigations on estimates of Hecke eigen Maass forms and holomorphic Hecke eigen cusp forms. Let $|\cdot|_{\mathrm{pet}}$ denote the point-wise Petersson norm on $\Sk$, which is given by equation \eqref{pet-norm}. 

One of the first investigations of estimates of holomorphic cusp forms was initiated by Jorgenson and Kramer in \cite{jk1}. Let $\G:=\Gamma_{0}(N)$, and let $\Sk$ denote the complex vector space of cusp forms of weight-$k$ with respect to $\G$. Let $\big\lbrace f_{1},\ldots, f_{\dk} \big\rbrace$ be an orthonormal basis of $\Sk$, with respect to Petersson inner-product on $\Sk$. For any $z,w\in\H$, the Bergman kernel associated to $\Sk$ is given by the following formula
\begin{align*}
\bk(z,w):=\sum_{j=1}^{\dk}f_{j}(z)\overline{f_{j}(w)}. 
\end{align*}

Then, in 2004, for $k=2$, Jorgenson and Kramer established the following estimate
\begin{align*}
\big|\mathcal{B}_{\G}^{2}(z,z)\big|_{\mathrm{pet}}:=y^{2}\big|\mathcal{B}_{\G}^{2}(z,z)\big|=O(1),
\end{align*}
and the implied constant is independent of the level of  the arithmetic subgroup $\Gamma_{0}(N)$. 

Furthermore, in 2013, in \cite{jk2}, Friedman, Jorgenson, and Kramer, extended their results to higher weights. For $\G$ any cocompact Fuchsian subgroup, they proved the following estimate
\begin{align*}
\big|\bk(z,z)\big|_{\mathrm{pet}}=O_{\G}(k);
\end{align*}
and when $\G$ is noncompact, they proved the following estimate
\begin{align*}
\big|\bk(z,z)\big|_{\mathrm{pet}}=O_{\G}\big(k^{\frac{3}{2}}\big).
\end{align*}

Moreover, the implied constants in both the above estimates were proved to be stable in covers.
 
In 2007, for $ f \in \mathcal S^k (\mathrm{SL}_{2}(\Z))$, a Hecke eigen cusp form, which is normalized with respect to the Petersson inner-product, for any $\epsilon>0$, in \cite{xia}, Xia proved the following estimate 
\begin{align}\label{xia}
\sup_{z\in \mathrm{SL}_2 (\Z) \backslash \H}\big|f(z)\big|_{\mathrm{pet}}=O_{\epsilon}\big(k^{\frac{1}{4}+\epsilon}\big),
\end{align}
where the implied constant depends only on the choice of $\epsilon>0$. Furthermore, the above estimate is proved to be optimal.  

In 2010, in \cite{blomer-holowinsky}, Blomer and Holowinsky have derived estimates of  holomorphic Hecke eigen cusp forms and Hecke eigen Maass forms associated to the arithmetic subgroup $\G_{0}(N)$, in terms of both the level aspect, and the weight aspect. Let $N\in\N$ be square-free, and let $f \in \mathcal S^k (\G_0 (N))$ be a Hecke eigen cusp form of weight-$k$, which is normalized with respect to the Petersson inner-product.  Blomer and Holowinsky then show that 
\begin{align}\label{blomer-holow}
\sup_{z\in \G_0 (N) \backslash \H}\big|f(z)\big|_{\mathrm{pet}}=O_{k}\big(N^{-\frac{1}{37}}\big),
\end{align}
where the implied constant depends only on $k$. The above estimate also holds true for any Hecke eigen Maass form, which is $L^2$-normalized, and the implied constant depends only on $\lambda$, the Laplacian eigenvalue of the Maass form. 

With hypothesis as above, for any $\epsilon>0$, Blomer and Holowinsky made the following conjecture 
\begin{align*}
\sup_{z\in \G_0 (N) \backslash \H}\big|f(z)\big|_{\mathrm{pet}}=O_{\epsilon}\big(k^{\frac{1}{4}+\epsilon}\cdot N^{-\frac{1}{2}+\epsilon}\big)
\end{align*}
where the implied constant depends only on the choice of $\epsilon>0$.

We now introduce our main result. Sarnak's conjecture on Hecke eigen Maass forms is also expected to hold true for holomorphic Hecke eigen cusp forms associated to a co-compact Fuchsian subroup $\G$, which is as described above. 

Let $f\in\Sk$ be a Hecke eigen cusp form, which is normalized with respect to the Petersson inner-product. Then, the holomorphic version of Sarnak's conjecture is the following estimate
\begin{align}\label{conj-sarnak}
\sup_{z\in X}\big| f(z)|_{\mathrm{pet}}=O_{\epsilon}\big(k^{\epsilon}\big),
\end{align}
where the implied constant depends only on the choice of $\epsilon>0$.

For any general $f\in\Sk$, which is normalized with respect to the Petersson inner-product, the following estimate is the convexity estimate
\begin{align*}
\sup_{z\in X}\big| f(z)|_{\mathrm{pet}}=O_{\G}\big(k^{\frac{1}{2}}\big),
\end{align*}
and if $f$ is not Hecke, then the above estimate is optimal. 

The first known estimate in the cocompact setting is derived by Das and Sengupta. Let $\G$ be a cocmpact Fuchsian subgroup associated to a quaternion division algebra $\cala$ defined over $\Q$ (see section \ref{subsec-2.1} for details). Let $f\in\Sk$ be a Hecke eigen cusp form, which is normalized with respect to the Petersson inner-product. Then, for any $\epsilon>0$, Das and Sengupta derived the following sub-convexity bound in \cite{das-sen-erratum}
\begin{align}\label{das-sen}
\sup_{z\in X}\big| f(z)|_{\mathrm{pet}}=O_{\cala,\epsilon}\big(k^{\frac{1}{2}-\frac{2}{131}+\epsilon}\big),
\end{align}   
where the implied constant depends on the quaternion algebra $\cala$, and on the choice of $\epsilon>0$. Das and Sengupta emulated and adapted the amplification technique from \cite{iwaniec-sarnak}, and combined it with estimates of the Bergman kernel, to derive the above estimate. The authors claim $O_{\G,\epsilon}\big(k^{\frac{1}{2}-\frac{1}{33}+\epsilon}\big)$ in \cite{das-sen}, however, after minor corrections, the authors eventually arrived at estimate \eqref{das-sen} in an erratum \cite{das-sen-erratum}, which is available on the first author's homepage.

\vspace{0.2cm}
Currently the best known estimate in the cocompact setting is derived by Khayutin and Steiner.  With hypothesis as above, using theta correspondence, for an any $\epsilon>0$, Khayutin and Steiner have proved the following estimate in \cite{steiner}
\begin{align}\label{steiner}
\sup_{z\in X}\big| f(z)|_{\mathrm{pet}}=O_{\cala,\epsilon}\big(k^{\frac{1}{4}+\epsilon}\big),
\end{align}  
which is as strong as the one derived by Xia (estimate \eqref{xia}) for $\mathrm{SL}_{2}(\mathbb{Z})$.

\subsection{Main result}\label{subsec-1.2}
We now state the main result of the article. Let $\G$ be a Fuchsian subgroup as in \cite{das-sen}, and let $f\in\Sk$ be a holomorphic Hecke eigen cusp form, which is normalized with respect to the Petersson inner-product  on $\Sk$. Then, for $k\gg 1$, and for any $\epsilon >0$, we have the following estimate
\begin{align}\label{main-result}
\sup_{z\in X}\big| f(z)|_{\mathrm{pet}}=O_{\cala,\epsilon}\big(k^{\frac{1}{2}-\frac{1}{12}+\epsilon}\big),
\end{align}
where the implied constant depends on the quaternion algebra $\cala$, and on the choice of $\epsilon>0$.  

The major difficulty in the co-compact situation is that, unlike in the case of congruence subgroups, cusp forms associated to a co-compact subgroups do not admit 
a Fourier expansion. Proofs of estimates \eqref{xia} and \eqref{blomer-holow}, heavily rely on the Fourier expansion of cusp forms.  

To overcome the absence of Fourier expansion of cusp forms in the co-compact setting, following \cite{das-sen} and \cite{iwaniec-sarnak}, we apply Jacquet-Langlands correspondence to utilize Deligne's bound on Hecke eigen values.  As in \cite{das-sen}, we adapt the amplification technique from \cite{iwaniec-sarnak}, and combine it with estimates of the Bergman kernel associated to the complex vector space $\Sk$, which we have explicitly computed in this article. Our estimates of the Bergman kernel are optimally derived, by adapting arguments from \cite{am1}. 

\vspace{0.2cm}
\begin{rem}
Estimate \eqref{steiner} is much stronger than our result \eqref{main-result}, of which the authors were unaware, until it was pointed out to them by Simon Marshall. However, the techniques involved in proving estimates \eqref{steiner} and \eqref{main-result} are very different. 
\end{rem}
\subsection*{Organization of the article}
In sections \ref{subsec-2.1} and \ref{subsec-2.2},  we set up the basic notation and background material. In section \ref{subsec-2.3}, we introduce Hecke operators, and in section \ref{subsec-2.4}, we introduce a counting function from \cite{iwaniec-sarnak}. In section \ref{sec-3}, using the counting function from \cite{iwaniec-sarnak} and combining it with arguments from \cite{das-sen}, we prove estimate \eqref{main-result}.  
\vspace{0.2cm}

\section{Hyperbolic Riemann surfaces, cusp forms, and Bergman kernel}\label{sec2}
In this section, we set up the notation, and state results from literature, which are used to prove estimate \eqref{main-result}. 
\vspace{0.2cm}

\subsection{Preliminaries}\label{subsec-2.1}
Let 
\begin{align*}
\H:=\big\lbrace z=x+iy\in\C\big|\, y>0 \big\rbrace
\end{align*}
denote the hyperbolic upper half-plane. Let $\hypx$ denote the hyperbolic metric on $\H$, which is the natural metric on $\H$, and is of constant negative curvature equal to 
$-1$. The hyperbolic metric $\hypx$, at the point $z=x+iy\in\H$, is given by the following formula
\begin{align*}
\hypx(z):=\frac{i}{2}\cdot\frac{dz\wedge d\overline{z}}{y^{2}}=\frac{dxdy}{y^{2}}.
\end{align*} 

Let $\dh(\cdot,\cdot)$ denote the natural distance function on $\H$, which is induced by the hyperbolic metric $\hypx$. Furthermore, we have the following relation
\begin{align}\label{dist-eqn}
\cosh^{2}\big(\dh(z,w)\slash 2\big)=\frac{|z-\overline{w}|^{2}}{4yv}=\frac{(x-u)^{2}+(y+v)^{2}}{4yv},
\end{align}
for $z=x+iy,\,w=u+iv\in\H$.

We now introduce the Fuchsian subgroup associated to a quaternion division algebra, as in \cite{iwaniec-sarnak}. For any $a,b\in\Z$, let 
$\cala:=\left(\frac{a,b}{\mathbb{Q}}\right)$ be a quaternion division algebra defined over $\Q$, with the usual norm operator $\caln_{\cala}$. We also assume that $\cala$ is split at infinity. Then, fix an embedding $\phi_{\cala}:\cala \longrightarrow\calm_{2}\big(\Q(\sqrt{a})\big)$. 

Let $\calr$ be a maximal order in $\cala$. Put
\begin{align*}
\calr(1):=\big\lbrace\alpha\in \cala|\,\caln(\alpha)=1 \big\rbrace,
\end{align*}
and let $\G:=\phi_{\cala}(\calr(1))\subset\mathrm{SL}_{2}(\R)$, which is a co-compact Fuchsian subgroup. The Fuchsian subgroup $\G$ acts on $\H$, via fractional linear transformations, and let $X:=\G\backslash \H$ denote the quotient space. The quotient space $X$ admits the structure of a compact hyperbolic Riemann surface.

The hyperbolic metric induces a metric on $X$, which is compatible with the natural complex structure of $X$, and we again denote it by $\hypx$. Furthermore, for $z,w\in X$, the geodesic distance between the points $z$ and $w$ on $X$, is given by $\dh(z,w)$. Set
\begin{align*}
\vx(X):=\int_{X}\hypx(x). 
\end{align*}

\subsection{Cusp forms and Bergman kernel}\label{subsec-2.2}
Let $\N$ denote the set of positive integers, and for $k\in\N$, let $\Sk$ denote the complex vector space of holomorphic cusp forms of weight-$k$, with respect to the subgroup $\G$, and let $\dk$ denote the dimension of $\Sk$, as a complex vector space. 

For odd $k$, as $-\mathrm{Id}\in\G$, $\Sk=\emptyset$, where $\mathrm{Id}$ denotes the $2\times 2$ identity matrix. 

For $f\in \Sk$, the point-wise Petersson norm at the point $w=u+iv\in \H$, is given by the following formula
\begin{align}\label{pet-norm}
\big|f(w)\big|_{{\rm{pet}}}:=v^{\frac{k}{2}}|f(w)|,
\end{align}
which is invariant with respect to the action of $\G$, and hence, defines a continuous function on the Riemann surface $X$. 

The point-wise Petersson norm induces an inner-product on $\Sk$. For $w=u+iv\in\H$, and $f,g\in \Sk$, the Petersson inner-product on $\Sk$, is given by the following integral
\begin{align*}
\langle f,g\rangle_{\rm{pet}}:=\int_{\calf}v^{k}f(w)\overline{g(w)}\hypx(w),
\end{align*}
where $\calf$ denotes a fixed fundamental domain for $\G$ in $\H$. 

Let $\lbrace f_{1},\ldots,f_{\dk}\rbrace$ denote an orthonormal basis of $\Sk$, with respect to the Petersson inner-product on $\Sk$. Then, for $z,w\in \H$, the Bergman kernel associated to $\Sk$, is given by the following formula
\begin{equation}\label{bkdefn}
\bk(z,w):=\sum_{j=1}^{\dk}f_{j}(z)\overline{f_{j}(w)}.
\end{equation}
By Riesz representation theorem, it follows that the definition of the Bergman kernel $\bk$, is independent of the choice of orthonormal basis of $\Sk$. 

For $z,w\in \H$, the Bergman kernel $\bk(z,w)$ is a holomorphic cusp form of weight-$k$ in the $z$-variable, and an anti-holomorphic cusp form of weight-$k$ in the $w$-variable. Hence, the point-wise Petersson norm of the Bergman kernel at the points $z=x+iy,\,w=u+iv\in \H$, is given by the following formula
\begin{align*}
\big|\bk(z,w)\big|_{\mathrm{pet}}=y^{\frac{k}{2}}v^{\frac{k}{2}}\big| \bk(z,w)\big|, 
\end{align*}
which is invariant with respect to the action of $\G$ on both the variables $z$ and $w$, and hence, defines a function on $X\times X$. 

Furthermore, $\bk$ is the generating function for the vector space $\Sk$, i.e., for $f\in \Sk$, and $z\in\H$ and $w=u+iv\in \calf$, we have
\begin{align*}
\int_{\calf}v^{k}\bk(z,w)f(w)\hypx(w)=f(z), 
\end{align*}
where as above, $\calf$ is a fixed fundamental domain of $X$.

For $k\in \N$ and $k\geq 3$, and $z,w\in \H$, the Bergman kernel $\bk(z,w)$ can also be represented, by the following infinite series (see Proposition $1.3$ on $p.\, 77$ in \cite{freitag})
\begin{align}
&\bk(z,w)=\frac{(k-1)(2i)^{k}}{4\pi}\sum_{\gamma\in \G}\frac{1}{\big( z-\overline{\gamma w}\big)^{k}}\cdot\frac{1}{\overline{j^{k}(\gamma,w)}}, \notag\\
&\mathrm{where\,\,for}\,\,\gamma=\left(\begin{matrix} a&b\\c&d\end{matrix}\right)\in \G,\,\,j(\gamma,w):=cw+d.\label{bkseries}
\end{align} 
The above formula for the Bergman kernel $\bk(z,w)$, which is given in \cite{freitag}, is missing a factor of $(2i)^{k}$, and that fact is taken into account in the above formula. 
\vspace{0.2cm}

\subsection{Hecke operators}\label{subsec-2.3}
Recall that we have a fixed embedding $\phi_{\cala}:\cala \longrightarrow\calm_{2}\big(\Q(\sqrt{a})\big)$. With notation as above, for any $n\in\N$, put
\begin{align*}
\calr(n):=\big\lbrace \alpha\in \calr|\,\caln_{\cala}(\alpha)=n\big\rbrace,
\end{align*} 
and set $\G(n):=\phi_{\cala}\big(\calr(n)\big)$.  It is well known that $\G$ acts on $\G(n)$ by left multiplication, and for any $\epsilon >0$, the cardinality of the set  $\G\backslash \G(n)$ is $O_{\epsilon}\big(n^{\epsilon}\big)$, where the implied constant depends only on the choice of $\epsilon >0$. 

We now recall the definition of Hecke action on the space of cusp forms. This action was originally constructed in \cite{eichler} using the theory of correspondences. For any $k,n\in\N$, and $f\in \Sk$, and $w\in \H$, define
\begin{align}\label{hecke-cusp}
\tnc f(w):=n^{\frac{k}{2}-1}\sum_{\gamma\in\G\backslash \G(n)}\frac{n^{\frac{k}{2}}}{j^{k}(\gamma ,w)}\cdot f(\gamma w).
\end{align}

A prime $\ell$ is called a ramified prime (or characteristic prime) for $\cala$, if $\cala_\ell := \cala \otimes \Q_\ell$ is a division algebra. It is well-known that the set of ramified primes is finite and its cardinality is even. There is an integer $q$ such that the Hecke operators $\tnc$, where $(n, q) = 1$  preserve the space of cusp forms and are self-adjoint operators on this space (see \cite{eichler}). This integer $q = q_1 q_2$, where $q_1$ is the product of characteristic primes $\ell$ for which $R_\ell = R \otimes \Z_\ell$ is a maximal order and $q_2$ is the product of primes $\ell$ such that $R_\ell$ is isomorphic to $\begin{pmatrix} \Z_\ell & \Z_\ell \\ \ell \Z_\ell & \Z_\ell \end{pmatrix}$. This is the same choice made in \cite[\S 2.2]{das-sen}.

With $n, q$ as above, we have $\tnc \tmc = \sum_{d | (m,n)} d \tmndc$.  Therefore, for any $k$ as above, we can choose a set of orthonormal basis for $\Sk$ with respect to the Petersson inner-product, such that all the basis elements are Hecke eigen cusp forms.

Furthermore, let $f\in\Sk$ be a Hecke eigen cusp form with a set of Hecke eigenvalues $\lbrace\lambda_{f}(n)\rbrace_{n\in\N}$. Then, from Jacquet-Langlands correspondence (see p. 470 and p. 494 in \cite{jac-lang} and \cite{hejhal}), there exists an integer $D \in \N$ which depends only on the maximal order $R$, such that there exists a cusp form  $F$ of weight-$k$ with respect to the arithmetic subgroup $\G_{0}(D)$, and the Hecke eigenvalues of $F$ coincide with that of $f$, for all $(n,D)=1$. Note that $D$ is divisible by all primes $q$ that are ramified primes for $\cala$, and no other primes.

Hence, from Deligne's celebrated  estimate, for all $(n,D)=1$, for any $\epsilon>0$, we have the following estimate
\begin{align}\label{deligne-est}
\big|\lambda_{f}(n)\big|=O_{\epsilon}\big(n^{\frac{k-1}{2}+\epsilon}\big),
\end{align}
where the implied constant depends on the choice of $\epsilon>0$. 

Furthermore, set 
\begin{align}\label{normalized-hecke}
\eta_{f}(n):=\frac{\lambda_{f}(n)}{n^{\frac{k-1}{2}}}.
\end{align}
For $m,n\in\N$ with $(m,D)=1$ and $(n,D)=1$, the normalized eigenvalues satisfy the following relation
\begin{align}\label{hecke-relns}
\eta_{f}(m)\cdot\eta_{f}(n)=\sum_{d\,\mid(m,n)}\eta_{f}\bigg(\frac{mn}{d^{2}}\bigg).
\end{align}

We now define Hecke action on continuous functions on $X$. Let $\calc(X)$ denote the space of continuous functions on $X$. For any $h\in\calc(X)$, and $w\in \H$, define
\begin{align}\label{hecke-l2}
\tnl h(w) :=  \sum_{\gamma\in\G\backslash \G(n)} h(\gamma w).
\end{align}

\subsection{Counting function}\label{subsec-2.4}
For any $n\in\N$, from  \cite{iwaniec-sarnak}, we now introduce a  counting function, associated to counting elements of $\G(n)=\phi_{\cala}\big(\calr(n)\big)$. With notation as above, for any  $n\in\N$ and $z\in \H$, 
let 
\begin{align*}
\cals_{\Gamma(n)}\big(z;\rho\big):=\big\lbrace \gamma |\,\gamma \in \G(n),\,u(z,\gamma z)=\sinh^{2}\big(\dh(z,\gamma z) \slash 2\big)\leq \rho \big\rbrace.
\end{align*}

\vspace{0.1cm}
Combining Lemma 1.3  and Remark 1.5 in \cite{iwaniec-sarnak} (see p. 307 and p. 310), for $z\in \H$, and for any $\epsilon>0$, we have the following  estimates
\begin{align}\label{iw-sar}
&\sum_{\gamma\in\cals_{\Gamma(n)}(z;\,n^{-3})}\frac{1}{\cosh^{k}\big(\dh(z,\gamma z)\slash 2\big)}=O_{\cala,\epsilon}\big( n^{\epsilon}\big)\notag\\[0.1cm]
&\sum_{\gamma\in\G(n)\backslash\cals_{\Gamma(n)}(z;\,n^{-3})}\frac{1}{\cosh^{k}\big(\dh(z,\gamma z)\slash 2\big)} \ll_{\cala,\epsilon} n^{1+\epsilon}
\int_{n^{-3}}^{\infty}\bigg(\frac{1}{(1+u)^{\frac{k}{2}}}+\frac{1}{u^{\frac{3}{4}}(1+u)^{\frac{k}{2}}}\bigg)du,
\end{align}  
where the implied constants in the above estimates depend on $\cals$, and on the choice of $\epsilon>0$.

\section{Proof of the estimate \eqref{main-result}}\label{sec-3} 
In this section, using estimates from previous sections, we prove estimate \eqref{main-result}. We emulate the proof of Theorem 1.1 from \cite{das-sen}, which itself is inspired from the proof of Theorem $0.2$ in \cite{iwaniec-sarnak}.  

\begin{thm}\label{thm}
With notation as above, let $f\in\Sk$ be a Hecke eigen cusp form, which is normalized with respect to the Petersson inner-product  on $\Sk$. Then, for $k\gg 1$, 
and for any $\epsilon >0$, we have the following estimate
\begin{align}\label{thm:eqn}
\sup_{z\in X}\big| f(z)|_{\mathrm{pet}}=O_{\cala,\epsilon}\big(k^{\frac{1}{2}-\frac{1}{12}+\epsilon}\big),
\end{align}
where the implied constant depends on the quaternion algebra $\cala$, and on the choice of $\epsilon>0$.  
\end{thm}
\begin{proof}
We now fix an orthonormal basis of Hecke eigen cusp forms $\big\lbrace f_{1}:=f,\ldots,f_{\dk}\big\rbrace$ for $\Sk$. For a fixed $N\in\N$, let 
\begin{align*}
\cals:=\big\lbrace\alpha_{1},\ldots,\alpha_{N}\big|\,n\in\N,\,\alpha_{n}=0,\,\mathrm{if}\,q|n\big\rbrace,
\end{align*}
be a set of constants, and $q$ is the prime associated to the order $\calr$.

For any $z=x+iy\in \H$, and from the definitions of the Bergman kernel, Hecke operator, and normalized Hecke eigenvalues from equations \eqref{bkdefn}, \eqref{hecke-cusp}, and \eqref{normalized-hecke}, respectively, and combining it with the Hecke relation \eqref{hecke-relns}, we derive
\begin{align}\label{thm:eqn1}
\sum_{j=1}^{\dk}\big|f_{j}(z)\big|^{2}\cdot\bigg|\sum_{\alpha_{n}\in\cals}\alpha_{n}\eta_{f_{j}}(n)\bigg|^{2}&=\sum_{\alpha_{m},\alpha_{n}\in\cals}\alpha_{n}\overline{\alpha_{m}}\bigg(\sum_{j=1}^{\dk}f_{j}(z)\overline{f_{j}(z)}\eta_{f_{j}}(m)\eta_{f_{j}}(n)\bigg)\notag\\[0.1cm]&=\sum_{\alpha_{m},\alpha_{n}\in\cals} \sum_{j=1}^{\dk}\alpha_{n}\overline{\alpha_{m}}\sum_{d|(m,n)}\eta_{f_{j}}\bigg(\frac{mn}{d^{2}}\bigg)\big(f_{j}(z)\overline{f_{j}(z)}\big)\notag\\&=
\sum_{\alpha_{m},\alpha_{n}\in\cals}\alpha_{n}\overline{\alpha_{m}}\sum_{d\,\mid(m,n)}\bigg(\frac{d^{2}}{mn}\bigg)^{\frac{k-1}{2}}\cdot\tmndcw\bk(z,z),
\end{align} 
where $\tmndcw$ denotes the Hecke operator for cusp forms acting on the second variable. For any $\gamma\in\G\big(mn\slash d^{2}\big)$, observe that
\begin{align*}
\Im(\gamma z)=\frac{mnv}{d^{2}\big|j(\gamma ,z)\big|^{2}}.
\end{align*}

Combining the above observation with the definition of Hecke operators for cusp forms and $L^{2}$ functions, which are as defined in equations \eqref{hecke-cusp} and \eqref{hecke-l2}, respectively, we deduce that
\begin{align}\label{thm:eqn2}
&\bigg|y^{k}\bigg(\frac{d^{2}}{mn}\bigg)^{\frac{k-1}{2}}\cdot\tmndcw\bk(z,z)\bigg|\leq\notag\\[0.12cm]&\sum_{\gamma\in\G\backslash \G(mn\slash d^{2})}\frac{d}{\sqrt{mn}}\cdot \big(y\Im (\gamma z)\big)^{\frac{k}{2}}\big|\bk(z,\gamma z)\big|=\frac{d}{\sqrt{mn}}\cdot \tmndlw\big|\bk(z, z)\big|_{\mathrm{pet}},
\end{align}
where $\tmndlw$ denotes the Hecke operator for cusp forms acting on the second variable. 

Using estimates \eqref{iw-sar}, we now estimate $\tmndlw\big|\bk(z, z)\big|_{\mathrm{pet}}$. For brevity of exposition, set 
\begin{align*}
\mathfrak{n}:=\frac{mn}{d^{2}}.
\end{align*}

For $z\in\mathbb{H}$, and for any $\epsilon >0$, from the definitions of the Hecke operator and that of the Bergman kernel from equations \eqref{bkseries} and  \eqref{hecke-l2}, respectively, and applying formula \eqref{dist-eqn} and estimate \eqref{iw-sar}, we compute
\begin{align}\label{thm:eqn3}
&\tmndlw\big|\bk(z, z)\big|_{\mathrm{pet}}=\sum_{\gamma\in \G\backslash \G(\mathfrak{n})}\big|\bk(z, \gamma z)\big|_{\mathrm{pet}}
\leq\notag\\[0.12cm]&\frac{k-1}{4\pi}\sum_{\gamma\in \G\backslash \G(\mathfrak{n})}\sum_{\gamma^{\prime}\in\G}\frac{1}{\cosh^{2}\big( \dh(z,\gamma^{\prime}\gamma z)\slash 2\big)}=\frac{k-1}{4\pi}\sum_{\gamma\in \G(\mathfrak{n})}\frac{1}{\cosh^{2}\big( \dh(z,\gamma z)\slash 2\big)}\ll_{\cala,\epsilon}
\notag\\[0.12cm]&\mathfrak{n}^{\epsilon}k+\mathfrak{n}^{1+\epsilon}k\int_{\mathfrak{n}^{-3}}^{\infty}\bigg(\frac{1}{(1+u)^{\frac{k}{2}}}+\frac{1}{u^{\frac{3}{4}}(1+u)^{\frac{k}{2}}}\bigg)du.
\end{align}

We now estimate the second term on the right hand-side of the above inequality. For $\mathfrak{n}\gg0$, we derive
\begin{align*}
\mathfrak{n}^{1+\epsilon}k\int_{\mathfrak{n}^{-3}}^{\infty}\bigg(\frac{1}{(1+u)^{\frac{k}{2}}}+\frac{1}{u^{\frac{3}{4}}(1+u)^{\frac{k}{2}}}\bigg)du\ll_{\cala,\epsilon}
\mathfrak{n}^{\frac{13}{4}+\epsilon} \bigg(\frac{\mathfrak{n}^{3}}{\mathfrak{n}^{3}+1}\bigg)^{\frac{k}{2}}=\mathfrak{n}^{\frac{13}{4}+\epsilon}
 \bigg(1-\frac{1}{\mathfrak{n}^{3}+1}\bigg)^{\frac{k}{2}}.
\end{align*}

Combining estimates \eqref{thm:eqn1}, \eqref{thm:eqn2}, and \eqref{thm:eqn3} with the above inequality, we arrive at the following inequality
\begin{align}\label{thm:eqn4}
y^{k}\sum_{j=1}^{\dk}\big|f_{j}(z)\big|^{2}\cdot\bigg|\sum_{\alpha_{n}\in\cals}\alpha_{n}\eta_{f_{j}}(n)\bigg|^{2}\leq
\sum_{\alpha_{m},\alpha_{n}\in\cals}\big|\alpha_{n}\overline{\alpha_{m}}\big|\sum_{d\,\mid(m,n)}
\frac{d}{\sqrt{mn}}\cdot\tmndlw\big|\bk(z,\gamma z)\big|_{\mathrm{pet}}\notag\\[0.12cm]\ll_{\cala,\epsilon} 
\sum_{\alpha_{m},\alpha_{n}\in\cals}\big|\alpha_{n}\overline{\alpha_{m}}\big|\sum_{d\,\mid(m,n)}\bigg(\bigg(\frac{d^{2}}{mn}\bigg)^{\frac{1}{2}-\epsilon}k+
 \left(\frac{mn}{d^{2}}\right)^{\frac{11}{4}+\epsilon} \bigg(1-\frac{1}{\big(mn\slash d^{2}\big)^{3}+1}\bigg)^{\frac{k}{2}} \bigg).
 \end{align}

From computations in the proof of Theorem 0.1(a) from  \cite{iwaniec-sarnak} (see p. 310), for any $\epsilon>0$, we have the following estimates for the two terms on the right hand-side of above inequality
\begin{align}
\sum_{\alpha_{m},\alpha_{n}\in\cals}\big|\alpha_{n}\overline{\alpha_{m}}\big|\sum_{d\,\mid(m,n)}\bigg(\frac{d^{2}}{mn}\bigg)^{\frac{1}{2}-\epsilon}k
 \ll_{\epsilon}N^{\epsilon}\bigg(\sum_{\alpha_{n}\in\cals}\big|\alpha_{n}\big|^{2}\bigg)k;\label{thm:eqn5}\\[0.12cm]
 \sum_{\alpha_{m},\alpha_{n}\in\cals}\big|\alpha_{n}\overline{\alpha_{m}}\big|\sum_{d\,\mid(m,n)}  \left(\frac{mn}{d^{2}}\right)^{\frac{11}{4}+\epsilon} \bigg(1-\frac{1}{\big(mn\slash d^{2}\big)^{3}+1}\bigg)^{\frac{k}{2}} \ll_{\epsilon}\notag\\[0.12cm]\bigg( \sum_{\alpha\in\cals}\big|\alpha_{n}\big|\bigg)^{2}N^{\frac{11}{2}+\epsilon}
 \bigg(1-\frac{1}{N^{3}+1}\bigg)^{\frac{k}{2}} .\label{thm:eqn6}
\end{align}

Now, as in \cite{das-sen} and \cite{iwaniec-sarnak}, we choose
\begin{align*}
\alpha_{n}:=\begin{cases}
\eta_{f}(p)\,\,\,&\mathrm{if}\ p\nmid D\ \mathrm{ and }\ n=p\leq\sqrt{N},\\
-1\,\,\,&\mathrm{if}\ p\nmid D\ \mathrm{ and }\,\,n=p^{2}\leq N,\\
0\,\,\,&\mathrm{otherwise}.
\end{cases}
\end{align*}

From the discussion in section \ref{subsec-2.3}, under the Jacquet-Langlands correspondence, there exists a cusp form $F$ of weight-$k$, with respect to the arithmetic  subgroup $\G_{0}(D)$, with $D$ as in the above section, such that Hecke eigenvalues of $F$ coincide with the Hecke eigenvalues of $f$, for all $(n,D)=1$. So from Deligne's estimate \eqref{deligne-est}
\begin{align}\label{thm:eqn7}
\big|\eta_{f}(n)\big|=O_{\epsilon}\big(n^{\epsilon}\big),
\end{align}
where the implied constant depends only on the choice of $\epsilon>0$. Furthermore, for primes $p\nmid D$, we have
\begin{align}\label{thm:eqn8}
\eta_{f}^{2}(p)-\eta_{f}(p^{2})=1.
\end{align}

Applying Deligne's estimate \eqref{thm:eqn7}, we have
\begin{align}
&\sum_{n\in\cals}\big|\alpha_{n}\big|^{2}\ll_{\epsilon} N^{\frac{1}{2}+\epsilon}\label{thm:eqn9}\\[0.12cm]&
\bigg( \sum_{\alpha\in\cals}\big|\alpha_{n}\big|\bigg)^{2}\ll_{\epsilon} N^{1+\epsilon}.\label{thm:eqn10}
\end{align}  

Combining  estimates \eqref{thm:eqn4}--\eqref{thm:eqn10}, we arrive at the following estimate
\begin{align*}
\big|f(z)\big|_{\mathrm{pet}}^{2} \cdot\bigg( \sum_{\substack{p\leq\sqrt{N}\\p\nmid q}}1\bigg)^{2}\ll_{\cala,\epsilon}N^{\frac{1}{2}+\epsilon}k+ N^{\frac{13}{2}+\epsilon}\bigg(1-\frac{1}{N^{3}+1}\bigg)^{\frac{k}{2}},
\end{align*}
which implies that 
\begin{align}\label{thm:eqn11}
\big|f(z)\big|_{\mathrm{pet}}^{2}\ll_{\cala,\epsilon} \frac{k}{N^{\frac{1}{2}-\epsilon}}+ N^{\frac{11}{2}+\epsilon}\bigg(1-\frac{1}{N^{3}+1}\bigg)^{\frac{k}{2}},
\end{align}

Substituting $12N^{3}\log N=k$, for $N\gg1$, we have the following the asymptotic
\begin{align*}
\bigg(1-\frac{1}{N^{3}+1}\bigg)^{\frac{k}{2}}=\bigg(1-\frac{1}{N^{3}+1}\bigg)^{6N^{3}\log N}=\frac{1}{N^{6}}+O\bigg(\frac{\log N}{N^{9}}\bigg).
\end{align*}

For red $k\gg 1$, substituting $12N^{3}\log N=k$, and the above asymptotic in estimate \eqref{thm:eqn11}, we arrive at the following estimate 
\begin{align*}
\big|f(z)\big|_{\mathrm{pet}}^{2} \ll_{\cala,\epsilon} N^{\frac{5}{2}+\epsilon}+\frac{1}{N^{\frac{1}{2}-\epsilon}}
 \ll_{\cala,\epsilon} k^{\frac{5}{6}+\epsilon}+\frac{1}{k^{\frac{1}{6}-\epsilon}}=O_{\cala,\epsilon}\big(k^{\frac{5}{6}+\epsilon}\big),
\end{align*}
which completes the proof of the theorem. 
\end{proof}
\vspace{0.2cm}

\subsection*{Acknowledgements}
 The first author acknowledges the support of INSPIRE research grant DST/INSPIRE/04/2015/002263 and the MATRICS grant MTR/2018/000636. The second author was partially supported by SERB grants EMR/2016/000840 and MTR/2017/000114. The second author would also like to thank the Max Planck Institute for Mathematics Bonn for their hospitality, where part of this work was carried out. 
 
Both the authors are very grateful to Soumya Das for several exchanges, and for patiently explaining the estimates derived in \cite{das-sen}. Lastly, the authors would also like to thank Simon Marshall, for communicating estimate \eqref{steiner} from \cite{steiner}.
\vspace{0.2cm}

\end{document}